\documentclass[12pt]{amsart}

\usepackage{amssymb,upref,enumerate}
\usepackage{tikz}
\usepackage{amsmath,amssymb,amscd,amsthm,amsxtra}
\usepackage{latexsym}
\usepackage{yfonts,mathrsfs, esint}
\usepackage{amsmath}
\usepackage{bbm}           
\usepackage{bm} 
\usepackage{eso-pic,graphicx}

\calclayout

\usepackage{pdfsync}

\usepackage[colorlinks=true, bookmarks=true,pdfstartview=FitV, linkcolor=blue, citecolor=blue, urlcolor=blue]{hyperref}

\def\Xint#1{\mathchoice 
{\XXint\displaystyle\textstyle{#1}}%
{\XXint\textstyle\scriptstyle{#1}}%
{\XXint\scriptstyle\scriptscriptstyle{#1}}%
{\XXint\scriptscriptstyle\scriptscriptstyle{#1}}%
\!\int} 
\def\XXint#1#2#3{{\setbox0=\hbox{$#1{#2#3}{\int}$} 
\vcenter{\hbox{$#2#3$}}\kern-.5\wd0}} 
 
\def\dashint{\Xint-}
\def\avgint{\Xint-}

\newcommand{\bey}{\begin{eqnarray*}}
\newcommand{\eey}{\end{eqnarray*}}
\newcommand{\ba}{\begin{align}}
\newcommand{\ea}{\end{align}}
\newcommand{\bea}{\begin{align*}}
\newcommand{\ena}{\end{align*}}
\newcommand{\be}{\begin{equation}}
\newcommand{\ee}{\end{equation}}
\newcommand{\R}{\mathbb R}

\newcommand{\M}{\mathcal M }

\newcommand{\bc}{\begin{center}}
\newcommand{\ec}{\end{center}}

\newcommand{\vf}{\vec{f}}
\newcommand{\vw}{\vec{w}}
\newcommand{\vp}{{\vec{p}}}

\newcommand{\supp}{\mathrm{supp}}

\newtheorem{theorem}{Theorem}[section]
\newtheorem{lemma}[theorem]{Lemma}
\newtheorem{corollary}[theorem]{Corollary}

\theoremstyle{definition}

\theoremstyle{remark}
\newtheorem{remark}[theorem]{Remark}

\numberwithin{equation}{section}
\allowdisplaybreaks

\begin{document}

\author{David Cruz-Uribe, OFS}
\address{David Cruz-Uribe, OFS \\ Department of Mathematics \\ University of Alabama \\
Tusca\-loosa, AL 35487, USA}
\email{dcruzuribe@ua.edu}

\author{Kabe Moen}
\address{Kabe Moen \\Department of Mathematics \\ University of Alabama \\
  Tuscaloosa, AL 35487, USA}
\email{kabe.moen@ua.edu}


\title[A multilinear reverse H\"older inequality]
{A multilinear reverse H\"older inequality with applications to
  multilinear weighted norm inequalities
}
\date{10/14/2017}

\begin{abstract} 
  In this note we prove a multilinear version of the reverse H\"older
  inequality in the theory of Muckenhoupt $A_p$ weights.  We give two
  applications of this inequality to the study of multilinear weighted
  norm inequalities.  First, we prove a structure theorem for
  multilinear $A_{\vec{p}}$ weights; second, we give a new sufficient
  condition for multilinear, two-weight norm inequalities for the
  maximal operator.
\end{abstract}

\thanks{The first author is supported by NSF Grant DMS-1362425 and research funds from the
  Dean of the College of Arts \& Sciences, the University of Alabama. The second author is supported by the Simons Foundation.}

\maketitle

\section{Introduction}
\label{section:intro}

The purpose of this note is to prove a multilinear version of the
reverse H\"older inequality in the theory of the Muckenhoupt $A_p$
classes.  We briefly recall the definitions; for complete details,
see~\cite{duoandikoetxea01}. For $1<p<\infty$, a weight
$w\in A_p$ if 
\[ [w]_{A_p} = \sup_Q \left( \avgint_Q w\,dx\right)
\left( \avgint_Q w^{1-p'}\,dx\right)^{p-1} < \infty; \]
here and below, the supremum is taken over all cubes in $\R^n$ with sides
parallel to the coordinate axes.  We denote the union of all the $A_p$
classes by $A_\infty$.   If $w\in A_\infty$, then it satisfies the
reverse H\"older inequality:  there exists $s>1$ such that
\[ [w]_{RH_s} = \sup_Q \left( \avgint_Q w^s\,dx\right)^{\frac{1}{s}}
\left( \avgint_Q w\,dx\right)^{-1} < \infty.  \]

Our main result extends this condition to $m$-tuples of $A_\infty$
weights.

\begin{theorem} \label{thm:multRH}
Given $1<s_1,\ldots,s_m<\infty$ such that $\sum\frac{1}{s_i}=1$, suppose $w_i \in RH_{s_i}$.  Then for all cubes $Q$,
\begin{equation} \label{eqn:multRH1}
 \prod_{i=1}^m \left(\avgint_Q w_i^{s_i}\,dx\right)^{\frac{1}{s_i}}
\leq C\avgint_Q \prod_{i=1}^m w_i\,dx. 
\end{equation}
\end{theorem}

The proof requires the rescaling properties of Muckenhoupt weights
from~\cite[Theorems~2.1, 2.2]{cruz-uribe-neugebauer95} which we state
here as a lemma.  (Also see~\cite{Johnson:1987wt}.)

\begin{lemma} \label{lemma:jn}  
Given a weight $w$:
\begin{enumerate}
\item $w\in RH_s$, $1<s<\infty$, if and only if $w^s \in A_\infty$;

\item $w\in A_p\cap RH_s$, $1<p,\,s<\infty$ if and only if $w^s \in
  A_q$, $q=s(p-1)+1$.   
\end{enumerate}
\end{lemma}

\begin{proof}
  By Lemma~\ref{lemma:jn}, for each $i$, since $w_i\in RH_{s_i}$, $w_i^{s_i} \in A_\infty$.
  Moreover, there exists $0<r<1$ such that
  $w_i^{rs_i} \in A_2\cap RH_{\frac{1}{r}}$; since the $A_p$ classes
  are nested ($A_p\subset A_q$ if $p<q$) we can find a
  single value of $r$ that works for all
  $i$.  Therefore,
  if we apply the reverse H\"older inequality, the $A_2$ condition,
  multilinear H\"older's inequality with exponents $s_i$, and then
  H\"older's inequality twice more, we
  get that for every cube $Q$,
\begin{multline*}
\prod_{i=1}^m \left(\avgint_Q w_i^{s_i}\,dx\right)^{\frac{1}{s_i}}
 \leq C \prod_{i=1}^m \left(\avgint_Q w_i^{rs_i}\,dx\right)^{\frac{1}{rs_i}} 
 \leq C \prod_{i=1}^m \left(\avgint_Q w_i^{-rs_i}\,dx\right)^{-\frac{1}{rs_i}}   \\
 \leq C\bigg(\avgint_Q \prod_{i=1}^m w_i^{-r}\,dx\bigg)^{-\frac{1}{r}} 
 \leq C\bigg(\avgint_Q \prod_{i=1}^m w_i^{r}\,dx\bigg)^{\frac{1}{r}}
 \leq C\avgint_Q \prod_{i=1}^m w_i\,dx.
\end{multline*}
\end{proof}

\begin{remark}
In the bilinear case, a version of this result was proved
in~\cite[Theorem~2.6]{cruz-uribe-neugebauer95}.  However, our proof
here is considerably simpler.  After this paper was submitted, we
learned that a similar result had been proved independently by Xue and Yan~\cite{Xue:2012jw}.
\end{remark}

Theorem~\ref{thm:main-thm} has two corollaries.  The first when $p=2$
is part of the folklore of harmonic analysis, but we have not been
able to find a proof in the literature.  We include this result
since we prove a multilinear version below: see Theorem~\ref{thm:multAinfty}.

\begin{corollary}
For $1<p<\infty$, if a weight $w$ is such that $w,\,w^{1-p'}\in
A_\infty$, then $w\in A_p$.  
\end{corollary}

\begin{proof}
  Let $w_1=w^{\frac{1}{p}}$, $w_2=w^{-\frac{1}{p}}$, and $s_1=p$,
  $s_2=p'$.  Then by Lemma~\ref{lemma:jn}, $w_i \in RH_{s_i}$, and
  inequality~\eqref{eqn:multRH1} becomes the $A_p$ condition.
\end{proof}

Our second corollary is a version of Theorem~\ref{thm:multRH}
that is particularly useful for applications to multilinear weights.

\begin{corollary}\label{cor:multRH} 
  Let $1<p_1,\ldots,p_m<\infty$, $0<p<\infty$ be such that
  $\sum_{i=1}^m\frac1{p_i}=\frac1p$, and let $\vw=(w_1,\ldots,w_m)$ be
  a vector of weights.  If $w_i\in A_\infty$, $1\leq i\leq m$, 
  then for all cubes $Q$,
$$\prod_{i=1}^m \left(\avgint_Qw_i\,dx\right)^{\frac{p}{p_i}}\leq C
\avgint_Q \prod_{i=1}^m w_i^{\frac{p}{p_i}}\,dx. $$
\end{corollary}

\begin{proof}
  By Lemma~\ref{lemma:jn}, for
  each $i$, since $w_i \in A_\infty$ we have
  $w_i^{\frac{p}{p_i}}\in RH_{\frac{p_i}{p}}$.  The desired inequality
  now
  follows from  Theorem~\ref{thm:multRH} with $s_i=\frac{p_i}{p}$.
\end{proof}

\bigskip

We now give three applications of Theorem~\ref{thm:multRH} (or more
properly, Corollary~\ref{cor:multRH}) to the theory of multilinear
weights.  We first consider the structure of multilinear $A_\vp $
weights.  To put our result into context, we briefly sketch the
history of multilinear weighted norm inequalities.  A multilinear
Calder\'on-Zygmund singular integral operator $T$ is an $m$-linear
operator with kernel $K$ such that
\[ T(\vf)(x) = \int_{\R^{mn}} K(x,y_1,\ldots,y_m) \prod_{i=1}^m f_i(y_i)\,dy_1\ldots dy_m, \]
where $\vf=(f_i,\ldots,f_m)$, $f_i \in C^\infty_c(\R^n)$ and $x\not\in \bigcap_{i=1}^m \supp(f_i)$, and the kernel satisfies certain regularity estimates.  (Precise definitions are not necessary for our purposes; see Grafakos and Torres~\cite{MR1880324} for details and further references.)
These operators satisfy the following $L^p$ space estimates:  if $1<p_1,\ldots,p_m<\infty$ and
$\frac{1}{p}=\sum\frac{1}{p_i}$, then
\[ \|T(\vf)\|_{L^p}\leq C\prod_{i=1}^m \|f_i\|_{L^{p_i}}. \]

Weighted norm inequalities for multilinear singular integral operators were first considered by Grafakos and Torres~\cite{MR1947875}.  They showed that if $p_0=\min(p_1,\ldots,p_m)$ and $w\in A_{p_0}$, then 
\[ \|T(\vf)\|_{L^p(w)}\leq C\prod_{i=1}^m \|f_i\|_{L^{p_i}(w)}. \]
This result was generalized by Grafakos and Martell~\cite{MR2030573}, who showed that if $w_i\in A_{p_i}$ and $w= \prod w_i^{\frac{p}{p_i}}$, then
\begin{equation} \label{eqn:wtd-mult}
\|T(\vf)\|_{L^p(w)}\leq C\prod_{i=1}^m \|f_i\|_{L^{p_i}(w_i)}. 
\end{equation}

The optimal class of weights for which \eqref{eqn:wtd-mult} holds was found by Lerner, {\em et al.}~\cite{MR2483720}.  Given $\vp=(p_1,\ldots,p_m)$, $\vw=(w_1,\ldots,w_m)$ and $w= \prod w_i^{\frac{p}{p_i}}$,  $\vw\in A_\vp$ if 
\[ [\vw]_{A_\vp}= \sup_Q \left(\avgint_Q w\,dx\right)^{\frac{1}{p}}
\prod_{i=1}^m \left(\avgint_Q w_i^{1-p_i'}\,dx\right)^{\frac{1}{p_i'}} <\infty. \]
It follows at once by H\"older's
inequality that if $w_i\in A_{p_i}$, then $\vw \in A_\vp$.  Such
weights are referred to as product $A_\vp$ weights.\footnote{If the
  weights $w_i$ are taken as prior, then this name makes sense.  If
  the weight $w$ is taken as prior, then it makes more sense to refer
  to these weights as ``factored'' $A_\vp$ weights.  However,
  ``product weights'' has become the standard terminology.}  However,
in~\cite{MR2483720} they constructed examples of weights
$\vw\in A_\vp$ such that the weights $w_i \not\in L^1_{\mathrm{Loc}}$
and so not in $A_{p_i}$.

Here we prove that if we impose some additional regularity on the
weights $w_i$, then the weight $\vw$ is a product weight.

\begin{theorem} \label{thm:main-thm}
Given $1<p_1,\ldots,p_m<\infty$ and $\frac{1}{p}=\sum\frac{1}{p_i}$, suppose $\vw=(w_1,\ldots,w_m)\in A_\vp$.  If $w_i\in A_\infty$, $1\leq i \leq m$, then $w_i\in A_{p_i}$.  
\end{theorem}

\begin{proof}[Proof of Theorem~\ref{thm:main-thm}]
By H\"older's inequality, for every $i$ and every cube $Q$,
\[ 1 \leq \left(\avgint_Q w_i\,dx\right)^{\frac{p}{p_i}}
\left(\avgint_Q w_i^{1-p_i'}\,dx\right)^{\frac{p}{p_i'}}. \]
Thus, if we fix a value of $i$,
\begin{multline*}
\bigg[\avgint_Q w_i\,dx
\left(\avgint_Q w_i^{1-p_i'}\,dx\right)^{p_i-1}\bigg]^{\frac{p}{p_i}}
= \left(\avgint_Q w_i\,dx\right)^{\frac{p}{p_i}}
\left(\avgint_Q w_i^{1-p_i'}\,dx\right)^{\frac{p}{p_i'}} \\
 \leq \prod_{i=1}^m \left(\avgint_Q w_i\,dx\right)^{\frac{p}{p_i}}
\left(\avgint_Q w_i^{1-p_i'}\,dx\right)^{\frac{p}{p_i'}}. 
\end{multline*}
By Corollary~\ref{cor:multRH}, 
\begin{multline*}
 \prod_{i=1}^m \left(\avgint_Q w_i\,dx\right)^{\frac{p}{p_i}}
\left(\avgint_Q w_i^{1-p_i'}\,dx\right)^{\frac{p}{p_i'}} \\
\leq C\bigg(\avgint_Q \prod_{i=1}^m w_i^{\frac{p}{p_i}}\,dx\bigg)
\prod_{i=1}^m \left(\avgint_Q w_i^{1-p_i'}\,dx\right)^{\frac{p}{p_i'}} 
\leq C[\vw]_{A_\vp}. 
\end{multline*}
If we combine these two inequalities, we see that $w_i \in A_{p_i}$.  
\end{proof}

\medskip

It is tempting to speculate that the condition $w_i\in A_\infty$ in
Theorem~\ref{thm:main-thm} can be weakened to $w_i \in
L^1_{\mathrm{Loc}}$.  This, however, is false, and we construct a
counter-example.  We note in passing that our example also gives an
example of a weight $\vec{w}\in A_{\vec{p}}$ that is not a product
weight; our construction is somewhat simpler than that given
in~\cite[Remark~7.2]{MR2483720}. 

For simplicity we consider the bilinear case $m=2$ on the real line and we will use the characterization $\vec{w}\in A_{\vec{p}}$ if and only if $w\in A_{2p}$ and $\sigma_i=w^{1-p_i'}\in A_{2p_i'}$ (see \cite[Theorem 3.6]{MR2483720}).  Let
\begin{equation*}
w_1(x) =  \frac{\chi_{[-1,1]}(x)}{|x|\log(e/|x|)^2}+ \chi_{[-1,1]^c}(x)\ \ \text{and} \ \ w_2(x)=1.
\end{equation*}
First, we have that $w=w_1^{\frac{p}{p_1}}$ belongs to $A_1\subset A_{2p}$. To see this, note that for $0<a<1$ and $b>0$, if we define
\[ v(x) = \frac{\chi_{[-1,1]}(x)}{|x|^a\log(e/|x|)^b} + \chi_{[-1,1]^c}(x), \]
then $v\in A_1$.  To show this, it suffices to check the $A_1$ condition on intervals $[0,t]$, $0<t<1$ (cf.~\cite{MR1406495}), but in this case it is straightforward to show that
\begin{equation}\label{eqn:logcalc} \int_0^t \frac{ds}{s^a\log(e/s)^b} \approx \frac{1}{t^{a-1}\log(e/t)^b}. \end{equation}
(Just compare the integrand to the derivative of the righthand side.)  Since $p/p_1<1$ we have that $w\in A_1$.
Next we will show that the weight $\sigma_1=w_1^{1-p_1'}$ belongs to $A_{2p_1'}$.  To see this, consider the dual weight (dual in the class $A_{(2p_1)'}$) $\sigma_1^{1-(2p_1')'}$ which is given by
$$\sigma_1(x)^{1-(2p_1')'}=w_1(x)^{(p_1'-1)((2p_1')'-1)}=w_1(x)^{\frac{p_1'-1}{2p_1'-1}}.$$
Again, since $\frac{p_1'-1}{2p_1'-1}<1$ the calculation \eqref{eqn:logcalc} shows that $\sigma_1^{1-(2p_1')'}\in A_1\subset A_{(2p_1')'}$ which implies $\sigma_1\in A_{2p_1'}$.  This shows $\vw \in A_{\vp}$.  However,
$w_i\not\in A_{p_i}$: since
$w_1 \in L^1_{\mathrm{Loc}}\setminus \bigcup_{p>1} L^p_{\mathrm{Loc}}$
it cannot satisfy a reverse H\"older inequality and so is not in
$A_\infty$.

\begin{remark}
It would be interesting to have an intrinsic characterization of those
$A_{\vec{p}}$ weights that can be factored as product weights.  
\end{remark}

\medskip

The second application of Corollary~\ref{cor:multRH} relates to the
characterization of $A_{\vec{p}}$ weights in \cite{MR2483720} that we
used above.  They proved that $\vec{w}\in A_{\vec{p}}$ if and only if
$w\in A_{mp}$ and $\sigma_i\in A_{mp_i'}$.  We can relax this assumption.

\begin{theorem} \label{thm:multAinfty} 
  Suppose $1<p_1,\ldots,p_m<\infty$, $\frac1p=\sum_i\frac1{p_i}$ and
  $\vec{w}=(w_1,\ldots,w_m)\in A_{\vec{p}}$ is a vector weight with
  $w=\prod_i w^{\frac{p}{p_i}}$ and $\sigma_i=w_i^{1-p_i'}$.  Then
  $\vec{w}=(w_1,\ldots,w_m)\in A_{\vec{p}}$ if and only if
  $w,\sigma_1,\ldots,\sigma_m\in A_\infty$.
\end{theorem}

\begin{proof} 
  It is enough to show that if $w,\sigma_1,\ldots,\sigma_m$ belong to
  $A_\infty$, then $\vec{w}\in A_{\vec{p}}$.  Notice that
$$\frac{1}{p}+\sum^m_{i=1}\frac{1}{p_i'}=m.$$
Then, for any cube $Q$ we have
\begin{multline*}
\left(\dashint_Q w\,dx\right)^{\frac1p}\prod_{i=1}^m\left(\dashint_Q \sigma_i\,dx\right)^{\frac{1}{p_i'}}=\left[\left(\dashint_Q w\,dx\right)^{\frac1{mp}}\prod_{i=1}^m\left(\dashint_Q \sigma_i\,dx\right)^{\frac{1}{mp_i'}}\right]^m \\
\leq C\left(\dashint_Q w^{\frac{1}{mp}}\sigma_1^{\frac{1}{p_i'm}}\cdots \sigma_m^{\frac{1}{p_m'm}}\,dx\right)^m=C.
\end{multline*}
\end{proof}

Our final application of Corollary~\ref{cor:multRH} is to
two-weight multilinear norm inequalities.  In
\cite{MR2483720}, the authors introduced a new multilinear maximal function as a tool to prove the estimates \eqref{eqn:wtd-mult}:
$$\M(\vf)(x)=\sup_{Q\ni x} \prod_{i=1}^m \left( \, \dashint_Q |f(y_i)|\,dy_i\right).$$
They proved that 
\[ \|\M(\vf)\|_{L^p(w)} \leq C\prod_{i=1}^m \|f_i\|_{L^{p_i}(w_i)} \]
if and only if $\vec{w}\in A_{\vec{p}}$.  This result led naturally to
the question of proving two-weight inequalities for $\M$:  more
precisely, inequalities of the form
\begin{equation}\label{eqn:twoweightM}
  \|\M(\vec{f\sigma})\|_{L^p(u)}\leq C\prod_{i=1}^m \|f_i\|_{L^{p_i}(\sigma_i)},\end{equation}
where $\vec{f\sigma}=(f_1\sigma_1,\ldots,f_m\sigma_m)$.
Chen and Dami\'an \cite{MR3118310} defined a
multilinear analog of the Sawyer testing condition for the
Hardy-Littlewood maximal operator:  
$(u,\vec{\sigma})\in S_\vp$ if there exists a constant $C$ such that
for all cubes $Q$,
\begin{equation*}
\left(\int_Q \M(\vec{\sigma}\chi_Q)^p u\,dx\right)^{\frac1p}\leq
C\prod_{i=1}^m \sigma_i(Q)^{\frac1{p_i}}.
\end{equation*} 
The $S_{\vp}$ condition is necessary for the two weight norm
inequality \eqref{eqn:twoweightM}, but it is not known if
it is also sufficient.  Chen and Dami\'an \cite[Theorem
3.2]{MR3118310} proved that it is sufficient if they also assume that
the weights satisfy a multilinear reverse H\"older
condition: for all cubes $Q$,
\begin{equation}\label{eqn:multiRH}
\prod_{i=1}^m \left(\avgint_Q\sigma_i\,dx\right)^{\frac{p}{p_i}}\leq C
\avgint_Q \prod_{i=1}^m\sigma_i^{\frac{p}{p_i}}\,dx.
 \end{equation}
 Therefore, by Corollary~\ref{cor:multRH}, if we assume the weights
 $\sigma_i$ are in $A_\infty$, then we get that the condition $S_\vp$ is sufficient the
 two weight norm inequality \eqref{eqn:twoweightM}.

\begin{theorem}\label{thm:twoweightresultM} 
Suppose $u,\sigma_1,\ldots,\sigma_m$ are weights, $1<p_1,\ldots,p_m<\infty$, and $p$ is defined by $\frac1p=\sum_{i=1}^m\frac{1}{p_i}$.  If $\sigma_i\in A_\infty$ and $(u,\vec{\sigma})\in S_\vp$, then the two-weight norm inequality \eqref{eqn:twoweightM} holds.
\end{theorem}

\begin{remark}
  It is an open question to characterize the $m$-tuples of weights for
  which~\eqref{eqn:multiRH} holds.  A version of this
  question was posed in~\cite[Theorem~2.7]{cruz-uribe-neugebauer95};
  they showed that when $m=2$, if~\eqref{eqn:multRH1} holds
  and $w_1\in RH_{s_1}$, then $w_2 \in RH_{s_2}$.
\end{remark}

\begin{remark}
For an additional application of Corollary~\ref{cor:multRH},
see~\cite{LC-DCU}, where it is used to prove necessary conditions on
$b$ for the bilinear commutators $[b,T]_i$ to satisfy weighted norm
inequalities, where $T$ is a bilinear
Calder\'on-Zygmund singular integral or a bilinear fractional integral
operator. 
\end{remark}

\bibliographystyle{plain}
\bibliography{multAp}

\end{document}